\documentclass{amsart}
\usepackage{amsmath}
\usepackage{graphicx}
\usepackage{amsfonts}
\usepackage{amssymb}
\usepackage{pdfsync}
\usepackage{color}
\usepackage{mathrsfs}
\usepackage{epsfig}
\usepackage{url}
\usepackage{mathrsfs,a4wide}
\usepackage{url}
\usepackage{calligra}
\makeatletter
\def\@splitop#1#2\@nil{$\mathscr{#1}\!\!$\calligra#2\,\,}
\newcommand*\DeclareCursiveOperator[2]{%

  \newcommand#1{\mathop{\mbox{\@splitop#2\@nil}}\nolimits}}
\makeatother
\DeclareCursiveOperator{\Bnew}{B}
\DeclareCursiveOperator{\Cnew}{C}
\DeclareCursiveOperator{\Dnew}{D}
\DeclareCursiveOperator{\defe}{Def}

\theoremstyle{plain}
\newtheorem{theorem}{Theorem}[section]
\newtheorem{lemma}[theorem]{Lemma}
\newtheorem{proposition}[theorem]{Proposition}

\newtheorem{remark}[theorem]{Remark}

\theoremstyle{definition}
\numberwithin{equation}{section}

\newcommand{\A}{{\mathscr A}}

\newcommand{\F}{{\mathcal F}}

\newcommand{\G}{\mathcal{G}}

\newcommand{\el}{\mathrm{el}}

\newcommand{\R}{{\mathbb R}}
\newcommand{\N}{{\mathbb N}}
\newcommand{\Z}{{\mathbb Z}}

\newcommand{\es}{\mathrm{e}}

\newcommand{\ud}{\,\mathrm{d} }

\newcommand{\Tr}{{\mathrm{tr}}}

\newcommand{\Hp}{\mathsf{H}}

\newcommand{\e}{\varepsilon}

\newcommand{\f}{\varphi}

\newcommand{\newatop}{\genfrac{}{}{0pt}{1}}

\newcommand{\loc}{{\mathrm{loc}}}
\title
[Two slope functions minimizing fractional seminorms]{Two slope functions minimizing fractional seminorms and applications to misfit dislocations}

\author[L. De Luca]
{Lucia De Luca}
\address[L. De Luca]{Istituto per le Applicazioni del Calcolo ``M. Picone'', IAC-CNR, Roma (Italy)}
\email[L. De Luca]{lucia.deluca@cnr.it}

\author[M. Ponsiglione]
{Marcello Ponsiglione}
\address[M. Ponsiglione]{Dipartimento di Matematica ``G. Castelnuovo'', Sapienza Universit\`a di Roma, Roma (Italy)} 
\email[M. Ponsiglione]{ponsigli@mat.uniroma1.it}

\author[E. Spadaro]
{Emanuele Spadaro}
\address[E. Spadaro]{Dipartimento di Matematica ``G. Castelnuovo'', Sapienza Universit\`a di Roma, Roma (Italy)} 
\email[E. Spadaro]{spadaro@mat.uniroma1.it}

\usepackage{hyperref}
\begin{document}
\vskip .2truecm

\begin{abstract}
\small{
We consider periodic piecewise affine functions, defined on the real line, with two given slopes and prescribed length scale of the regions where the slope is negative.  We prove that, in such a class,  the minimizers of $s$-fractional Gagliardo seminorm densities, with $0<s<1$, are in fact periodic with the minimal possible period determined by the prescribed slopes and length scale.

Then, we determine the asymptotic behavior of the energy density as the ratio between the length of the two intervals where the slope is constant vanishes.

Our results, for $s=\frac 1 2$, have  relevant applications to the van der Merwe theory of misfit dislocations at semi-coherent straight interfaces. We consider  
two elastic materials having different elastic coefficients
 and casting parallel lattices having different spacing.
 As a byproduct of our analysis,
 we prove the periodicity of   optimal dislocation configurations  and
% that the optimal configuration of misfit dislocations is periodic; then, 
we provide  the sharp asymptotic energy density in the semi-coherent limit as the ratio between the two lattice spacings tends to one.  
 %This proves  that the optimal distribution of misfit dislocations is periodic, and generalizes the formula obtained in \cite{VdM} to the case of  two materials having different elastic moduli.

\vskip .3truecm \noindent \textsc{Keywords}: Fractional Seminorms; Periodic Minimizers; Misfit Dislocations
\vskip.1truecm \noindent \textsc{Mathematics Subject Classification}: 
74N05, % Mechanics of deformable solids; Phase transformations in solids; Crystals in solids
74N15, % Mechanics of deformable solids; Phase transformations in solids; Analysis of microstructure in solids
35R11 %Partial differential equations;  Miscellaneous topics in partial differential equations; Fractional partial differential equations
}
\end{abstract}
\maketitle

%\vskip -.5truecm
{\small \tableofcontents}

\section*{Introduction}
This paper deals with the emergence of periodic patterns arising from the competition between energy functionals favoring rapid oscillations, and penalizations/constraints  fixing the length scale of such oscillations. Specifically, we prove that  minimizers of  fractional seminorms among two slope functions with suitably prescribed 
length scales are periodic with the minimal allowed period.  

Periodicity for minimizers of the local $L^2$ norm has been provided in \cite{M93} in the case where the two slopes are equal in modulus and have opposite sign, and for general (not necessarily equal in modulus) opposite slopes in \cite{RW}. The case of fractional seminorms has been considered in \cite{GM}, again for equal (opposite) slopes. Antiferromagnetic energies have been treated in \cite{LG}.
% (see also \cite{DR,GR} for related models in higher dimension). 

In this  paper we deal with  general two slope functions  minimizing  $s$-fractional  seminorms; this setting is relevant for modeling misfit dislocations at semi-coherent interfaces.
We settle our analysis in the most comfortable framework where the admissible functions $u$ are piecewise affine with  fixed   slopes $1$ and $-\Lambda<0$, and the length scale  $\delta>0$  is prescribed by assuming that the region where the slope is $-\Lambda$ is given by  union of intervals of length $\delta$ whose interiors are mutually disjoint.  In this respect, our framework is much more rigid and easier to be analyzed with respect to the phase field based models quoted above. 
% and allows as to compute the surface energy density for semi-coherent interfaces separating two half-planes with different elastic moduli, filled with square lattices with different lattice spacing, generalizing in this way the energy  

%More precisely, we consider piecewise affine functions $u$ whose slope  $u'$ can take only two values: one positive and one negative. Moreover, 
In order to define finite energy densities, we also assume that $u$ is $T$ periodic for some $T>0$. The energy density of any admissible function satisfying the constraints above is given by the $s$-fractional Gagliardo seminorm $\F^T_s$ defined in \eqref{defF}. It is easily seen that such a functional favors rapid oscillations of $u$; prescribing only the two slopes $1$ and $ -\Lambda$,   the infimum of the energy would be zero and any minimizing sequence would converge  to zero locally uniformly. On the other hand, fixing the length scale $\delta$, we expect that minimizers are periodic functions with the minimal allowed period, which is clearly given by $(\Lambda + 1)\delta$. This is exactly the result provided by Theorem \ref{mainthm}. The proof of such a theorem does not rely on reflection positivity arguments used in \cite{GM}, that  seem to fail in our non-symmetric case; 
as well, the antiferromagnetic reflection techniques used in \cite{LG} would increase rather than decrease our fractional ferromagnetic type energies. 
Instead, our proof is based on easy first variation arguments and comparison principles. It is in this step that we take advantage of the rigid constraints on the space of our admissible functions. Our variations consist in moving to the right a $\delta$-interval $I:=(x_0-\delta, x_0)$, where $u'=-\Lambda$ and   $x_0$ is an absolute minimizer of $u$. Such a first variation yields that the average of the $s$-fractional Laplacian of $u$ on $I$ has a sign. Then, we superpose the graph of  a candidate periodic  minimizer $\bar u$ over  that of $u$ on $I$. In view of our rigid constraints, $\bar u \le u$, and exploiting fractional maximum principles we deduce that $\bar u = u$. We remark that if we replace the constraint that the region where $u'=-\Lambda$ is union of $\delta$-intervals with the weaker constraint that each of its connected components has length at least $\delta$, it would be not clear anymore that $\bar u\le u$. On the other hand, the fact that there are no constraints on the region where $u'=1$ leaves enough room to construct and compute continuous families of first variations. 

Then, we are interested in the asymptotic behavior of the energy densities as $\delta\to 0$ and $\Lambda \to +\infty$. This analysis relies on detecting the precise scales at which the $s$-fractional  seminorms (for $s \ge \frac 12$) concentrate. In this way the interaction between infinite intervals where the slope is constant can be reduced, up to lower order terms, to the interaction of a finite number of them, that could be computed explicitly (Proposition \ref{pers>}). 

Moreover, in Subsection \ref{estremi}, we develop a similar analysis for the extremal cases $s=0$ and $s=1$. 
Trivially, the $s$-fractional Gagliardo seminorms diverge as $s\to 1$; but
it is well known, since \cite{BBM}, that, after multiplying by $(1-s)$, they converge to the classical Dirichlet energy. Analogously, by \cite{MS},  the $s$-fractional Gagliardo seminorms multiplied by $s$ converge, as $s\to 0$,  to the squared $L^2$ norm. We refer the reader to \cite{CDKNP} for the corresponding $\Gamma$-convergence statements in the setting of compactly supported functions.
Here, without developing an asymptotic analysis of our functional as $s\to 0^+$ and $s\to 1^-$, we determine the minimizers of the functionals $\F^T_s$ among two slope functions for $s=0$ (corresponding somehow to the case treated in \cite{M93, RW}) and $s=1$ and we compute their asymptotic energy density.
  
In Section \ref{Sdue} we apply the results of  Section \ref{Suno} to the theory of misfit dislocations. 
Our starting point is the analysis developed by van der  Merwe in \cite{VdM}.  He considers semi-coherent straight interfaces between two  parallel square lattices with different spacing.  The two main assumptions in \cite{VdM} are the following. First, the top and bottom lattices behave as linearly elastic isotropic materials, with the same elastic moduli. Second, it is tacitly assumed, as a well understood and unanimously accepted truth (as in the celebrated Read-Shockley paper \cite{RS} for small angle grain boundaries),  that dislocations are periodically distributed along the interface. Under these assumptions,  the optimal transition profile is found using Fourier  analysis techniques and the  sharp energy density induced by such a uniform distribution of dislocations is computed. 

%We consider the analysis  of van der Merwe \cite{VdM} for semi-coherent straight interfaces, and we  introduce a simplified model. Van der Merwe assumes a periodic distribution of dislocations at the interface and then compute the optimal transition profiles and the  sharp asymptotic energy density induced by the misfit dislocations in the semi-coherent limit as the ratio between the upper and bottom lattice spacing tends tov one. 
Here we have a less ambitious goal, that is to  find the asymptotics of the energy density in the semi-coherent limit as the ratio between the top and bottom lattice spacings tends to one.
In fact, the detection of the optimal profile only gives  lower order corrections in   the asymptotics of the energy density. This motivated us to consider a simplified model where the transition is prescribed in a simple, non optimal way; namely, we consider rigid affine transitions. The resulting difference of the top and bottom traces of the displacement at the interface is a two slope function with a ``small'' positive slope, accommodating elastically the lattice misfits, and a ``big'' negative slope, providing the transition at the core length scale of the dislocation.  It is here that  we are led   to consider the case of two general (non equal in modulus) slopes.  Then, the elastic energy density induced by the resulting trace is given, up to pre-factors, by the $\frac 12$-fractional seminorm $\F^T_{\frac 12}$ defined in \eqref{defF}, and this  leads us to consider fractional rather than local energies. Then, as a consequence of our results in Section \ref{Suno}, we prove, rather than assume, the periodicity of the optimal  distribution of  dislocations. This is our main novelty with respect to the analysis in \cite{VdM} and, to the best of our knowledge, it is the first time that periodicity of optimal distributions of dislocations is proved in a simple but rigorous mathematical framework. 
A further step is that we deal with two half-planes having different elastic moduli, providing an explicit energy density depending on all the elastic moduli and on the lattice spacings of the two square lattices \eqref{gfvdm}. 

We highlight that a related model for semi-coherent interfaces has been introduced and analyzed by $\Gamma$-convergence in \cite{FPS}; there, only the asymptotic (in the semi-coherent limit) uniform distribution of dislocations is proved. Our analysis, up to minor adaptations,  provides the periodicity of optimal configurations of dislocations also for that model, completing the analysis in \cite{FPS}. 

Finally, let us mention that it may be worth to consider less rigid models based on phase field approximations such as those  in \cite{LG, GM, M93, RW}. Such problems, together with possible further developments, are discussed in Section \ref{Stre}.  

\vskip5pt
{\textsc Acknowledgements: LdL is a member of the Gruppo Nazionale per l'Analisi Matematica, la Probabilit\`a e le loro Applicazioni (GNAMPA) of the Istituto Nazionale di Alta Matematica (INdAM). E.~S.~has been supported by the ERC-STG grant 759229 {\sc HiCoS}. The authors are grateful to S. Daneri, M. Morini and E. Runa for interesting  discussions on the subject of this paper.}
%%%%%%%%%%%%%%%%%%%%%%%%%%%%%%%%%%%%%%%%%%%
%%%%%%%%%%%%%%%%%%%%%%%%%%%%%%%%%%%%%%%%%%%
%%%%%%%%%%%%%%%%%%%%%%%%%%%%%%%%%%%%%%%%%%%
\section{Periodicity of two slope functions with minimal fractional seminorms}\label{Suno}
%\subsection{Description of the problem}
%%%%%%%%%%%%%%%%%%%%%%%%%%%%%%%%%%%%%%%%%%%
%%%%%%%%%%%%%%%%%%%%%%%%%%%%%%%%%%%%%%%%%%%
%%%%%%%%%%%%%%%%%%%%%%%%%%%%%%%%%%%%%%%%%%%
In this section we introduce the fractional seminorm  energies and the class of admissible functions we will deal with. We will prove the periodicity of the minimizers and provide sharp energy densities.

\subsection{Periodicity of minimizers}  
Let $\Lambda,\, \delta > 0$. Let 
\begin{equation}\label{formT}
T=L (\Lambda +1) \delta \qquad \text{ for some } L\in \N \, . 
\end{equation}
We introduce the class of admissible functions $\A^{T,\Lambda,\delta}$ defined as
\begin{multline}\label{defA}
\A^{T,\Lambda,\delta}:=\Big\{u\in C^0(\R) \,:\,u\textrm{ is } T\textrm{-periodic and piecewise affine},\quad u'\in\{1,-\Lambda\}  \textrm{ a.e.},\\
(u')^{-1}(\{-\Lambda\})=\bigcup_{k\in\N}I^k
\textrm{ for some intervals }  I^k \textrm{ with } 
\\
\sharp (I^i \cap I^j) \le 1   \textrm{ for all } i\neq j, \,  |I^k|=\delta \textrm{ for all } k\in\N\},
\end{multline}
where $\sharp$ denotes the cardinality of a set and $|\cdot|$ is the  one dimensional Lebesgue measure.
Notice that  $L=1$ corresponds to the minimal possible period $T=(\Lambda +1)\delta$, namely, to functions alternating $\delta$-intervals where $u'= -\Lambda$ with $\Lambda\delta$-intervals where $u'=1$. 
We denote by  $\bar u$ one of such  $(\Lambda +1) \delta$-periodic functions; specifically, the one  that  on the interval $(-\delta,\Lambda \delta]$ is defined by 
\begin{equation}\label{defmin}
\bar u(x):= 
\begin{cases}
-\Lambda x & \textrm{if } -\delta < x<0,\\
x & \textrm{if } 0 \le x \le \Lambda \delta.
\end{cases}
\end{equation}
Trivially, $\bar u$ is also $L (\Lambda +1) \delta$ periodic for all $L\in\N$, so that  $\A^{T,\Lambda,\delta}\neq \emptyset$ for all choices of $T$ as in \eqref{formT}.

%Notice that, since $\frac{T}{(\Lambda +1) \delta} =L\in\N$, $\A^{T,\Lambda,\delta}\neq \emptyset$.

For every $0<s<1$ we consider the functional $\F^{T}_s :\A^{T,\Lambda,\delta}\to \R^+$ defined by
\begin{equation}\label{defF}
\F^{T}_s(u):=\frac{1}{2T}\int_{0}^T\ud x\int_{\R}\frac{|u(x)-u(y)|^2}{|x-y|^{1+2s}}\ud y.
\end{equation}

For every $u\in \A^{T,\Lambda,\delta}$,   the $s$-fractional Laplacian $(-\Delta)^s u$ of $u$ can be defined a.e. (in fact, everywhere except for the jump points of $u'$) as
\begin{equation*}
(-\Delta)^s u (x):=2\lim_{r\to 0^+}\int_{\R\setminus(x-r,x+r)}\frac{u(x)- u(y)}{|x-y|^{1+2s}}\ud y.
\end{equation*}
We notice that for any $u\in \A^{T,\Lambda,\delta}$ we have that $\F^{T}_s(u)$ is finite and that
$(-\Delta)^s u\in L^1_{\loc}(\R)$.

%Given $\Lambda>1$ we aim at minimizing the functional $\F^T_s$ defined in \eqref{defF} in the space
Let $u\in \A^{T,\Lambda,\delta}$ and let $x_0$ be a global minimizer of $u$. 
%Notice that $u'=-\Lambda$ on an interval $I$ with $\partial I= \{x_0-\delta, x_0\}$.  More precisely, 
We set
$I=(x_0-\delta, x_0)$ if  $x_0-\delta$ is a jump point for $u'$, while we set $I=[x_0-\delta, x_0)$ if 
$u'$ exists  at $x_0-\delta$. Clearly $u'=-\Lambda$ on $I$.
Possible competitors for $u$ are given by   functions $u_h$ satisfying (for $h>0$ small enough)  
$$
(u_h')^{-1}(\{-\Lambda\})= ((u')^{-1}(\{-\Lambda\}) \setminus {\{I + kT, \, k\in\Z\}}) \cup 
{\{(x_0 - \delta, x_0) + h + kT, \, k\in\Z\}}\, .
$$
%can be obtained by replacing moving to the right by $h>0$ (small enough) the points in $I+ kT$ for all $k\in\Z$. 
Such  competitors  can be clearly written by means of  additive variations as $u_h= u+h \varphi_h$ for some functions $\varphi_h$ satisfying $\varphi_h \to (\Lambda +1) \chi_{\{I + kT, \, k\in\Z\}}$ strongly in $L^1_{\loc}(\R)$ and weakly star in $L^\infty(\R)$  as $h\to 0^+$.
\begin{lemma}\label{L1}
Let $u\in \A^{T,\Lambda,\delta}$ and let, for $h>0$ small enough, 
$u_h= u+h \varphi_h \in \A^{T,\Lambda,\delta}$  for some functions $\varphi_h$ satisfying 
\begin{equation}\label{ffih}
\varphi_h \to (\Lambda +1) \chi_{\{I + kT, \, k\in\Z\}}
\end{equation}
strongly in $L^1_{\loc}(\R)$ and weakly star in $L^\infty(\R)$ as $h\to 0^+$, for some interval $I$. Then,
$$
\lim_{h\to 0} \frac{\F^{T}_s(u_h) - \F^{T}_s(u)}{h} = \frac{\Lambda +1}{T} \int_I (-\Delta)^s u(x) \ud x. 
$$
\end{lemma}

%%%%%%%%%%%%%%%
%%%%%%%%%%%%%%%
%%%%%%%%%%%%%%%
\begin{proof}
For every $h>0$ we have
\begin{equation}\label{computfirstvar}
\begin{aligned}
\frac{T\F^T_s(u+h\f_h)-T\F^T_s(u)}{h}=&hT\F^T_s(\f_h)+\int_{0}^T\ud x\int_{\R}\frac{(u(x)-u(y))(\f_h(x)-\f_h(y))}{|x-y|^{1+2s}}\ud y\\
=&hT\F^T_s(\f_h)+\int_{0}^T\f_h(x)\frac{1}{2}(-\Delta)^su(x)\ud x\\
&+\lim_{r\to 0^+}\int_{0}^T\ud x\int_{\R\setminus(x-r,x+r)}\f_h(y)\frac{u(y)-u(x)}{|x-y|^{1+2s}}\ud y\\
=&hT\F^T_s(\f_h)+\int_{0}^T\f_h(x)(-\Delta)^su(x)\ud x,
\end{aligned}
\end{equation}
where, in the last equality,  we have used Fubini Theorem and the change of variable $x'=x-kT$ and $y'=y-kT$ to deduce that
\begin{equation*}
\begin{aligned}
&\int_{0}^T\ud x\int_{\R\setminus(x-r,x+r)}\f_h(y)\frac{u(y)-u(x)}{|x-y|^{1+2s}}\ud y\\
=&\sum_{k\in\Z}\int_{0}^T\ud x\int_{(kT,(k+1)T)\setminus(x-r,x+r)}\f_h(y)\frac{u(y)-u(x)}{|x-y|^{1+2s}}\ud y\\
=&\sum_{k\in\Z}\int_0^T\ud y' \f_h(y') \int_{(-kT,(1-k)T)\setminus (y'-r,y'+r)}\frac{u(y')-u(x')}{|x'-y'|^{1+2s}}\ud x'\\
=&\int_{0}^T\ud y'\f_h(y')\int_{\R\setminus(y'-r,y'+r)}\frac{u(y')-u(x')}{|x'-y'|^{1+2s}}\ud x'. 
\end{aligned}
\end{equation*}
Furthermore, easy  estimates yield  $hT\F^T_s(\f_h) \to 0$ as $h\to 0^+$. Therefore, 
by taking the limit as $h\to 0^+$ in \eqref{computfirstvar} and using \eqref{ffih} we get the claim.
\end{proof}
%%%%%%%%%%%%%%%
%%%%%%%%%%%%%%%
%%%%%%%%%%%%%%%

Now we observe that the fractional Laplacian of the function $\bar u$ defined in \eqref{defmin} has zero average on the $\delta$-intervals with negative slope. In fact, since the map  $y\mapsto \bar u(y - \frac{\delta}{2}) - \bar u (-\frac{\delta}{2})$ is an odd function,  by an easy symmetry argument (namely, by a change of variable),  the following result holds true.
%%%%%%%%%%%%%%%
%%%%%%%%%%%%%%%
%%%%%%%%%%%%%%%
\begin{lemma}\label{L2}
The function $\bar u$ defined by \eqref{defmin} satisfies
$$
 \int_{-\delta}^0 (-\Delta)^s \bar u(x) \ud x=0.
$$
%$$
% \int_{\big(-\delta+k(\Lambda+1)\delta,k(\Lambda+1)\delta\big)} (-\Delta)^s u(x) \ud x=0
%$$
%for every $k\in\N$.
\end{lemma}
%%%%%%%%%%%%%%%
%%%%%%%%%%%%%%%
%%%%%%%%%%%%%%%
We are in a position to provide the main result of the paper, establishing the $(\Lambda + 1)\delta$-periodicity of the minimizers of  $\F^{T}_s$ in $\A^{T,\Lambda,\delta}$.  In the following results we will always assume $T$ of the form \eqref{formT}, the specific value of $L$  being irrelevant. In this respect, for instance, next theorem establishes that $L$ can be chosen equal to one. 
 
\begin{theorem}\label{mainthm}
For every $0<s<1$, the minimizers of $\F^{T}_s$ in $\A^{T,\Lambda,\delta}$ are the functions $\bar u(\cdot - x_0) + y_0$, with $x_0,\,y_0\in\R$, where $\bar u$ is defined in \eqref{defmin}.
\end{theorem}
%%%%%%%%%%%%%%%
%%%%%%%%%%%%%%%
%%%%%%%%%%%%%%%
\begin{proof}
Let $u$ be a minimizer of $\F^{T}_s$ in $\A^{T,\Lambda,\delta}$ and let $x_0$ be a global minimizer of $ u$. 
By minimality and by Lemma \ref{L1} we have 
\begin{equation}\label{eq1}
\int_{x_0 - \delta}^{x_0} (-\Delta)^s u(x) \ud x \ge 0. 
\end{equation}
Set $v(\cdot):= \bar u(\cdot-x_0) + u(x_0)$. Then $v\le u$ with $v\equiv u$ on $(-\delta,0)$. By Lemma \ref{L2} and \eqref{eq1} it follows that
$$
0= \int_{x_0 - \delta}^{x_0} (-\Delta)^s v(x) \ud x \ge  \int_{x_0 - \delta}^{x_0} (-\Delta)^s u(x) \ud x \ge 0. 
$$
Then, all the inequalities above are in fact equalities, from which we conclude $u\equiv v$.
\end{proof}

\subsection{Fractional energy densities}\label{calcoloden}
In Proposition \ref{pers<} and Theorem \ref{pers>} below we provide  the asymptotic behavior of the minimal energy density $\F^T_s$  in   $\A^{T,\Lambda,\delta}$  as $\delta\to 0$, $\Lambda\to +\infty$ with $\Lambda \delta=1$. This last condition is not restrictive and  it only fixes  the  oscillation of the admissible functions. In fact, exploiting the homogeneities of  $\F^T_s$, given any $u \in \A^{T,\Lambda,\delta}$ and $\eta>0$, the function $u^\eta$ defined by $u^\eta(x):=  \eta u(\frac{x}{\eta})$  belongs to $\A^{\eta T,\Lambda,\eta \delta}$, its oscillation is $\eta$ times the oscillation of $u$ and 
$\F^{\eta T}_s(u^\eta)= \eta^{2-2s} \F^T_s(u)$.
%%%%%%%%%%%%%%%%%%%%%%%%
%%%%%%%%%%%%%%%%%%%%%%%%
%%%%%%%%%%%%%%%%%%%%%%%%

We first consider the easy case $0< s< \frac 1 2$, which follows by the  dominated convergence theorem. 
\begin{proposition}\label{pers<}
Let $0< s< \frac 1 2$.
Let moreover $\{\delta_n\}_{n\in \N}$ be a positive vanishing sequence (as $n\to +\infty$); for all $n\in\N$, set $\Lambda_n:=\frac{1}{\delta_n}$, let $T_n$ be of the form \eqref{formT} (with $\delta$ replaced by $\delta_n$ and $\Lambda$ replaced by $\Lambda_n$)  and  let  
$u_n$ be a minimizer of $\F^{T_n}_{s}$ in $\A^{T_n,\Lambda_n,\delta_n}$. Then
$$
\lim_{n\to +\infty}  \F^{T_n}_{s}(u_n) = \frac 12 \int_{0}^1\int_\R \frac{|w(x)-w(y)|^2}{|x-y|^{1+2s}} \ud y \ud x ,
$$
where $w(x):= x - \lfloor x \rfloor$ is the so called mantissa function.
\end{proposition}

Now we focus on  the case $\frac 1 2\le s<1$; in view of our application to misfit dislocations, particularly relevant will be the critical case $s=\frac 12$.

\begin{theorem}\label{pers>}
Let $\frac 1 2\le s<1$.
Let moreover $\{\delta_n\}_{n\in \N}$ be a positive vanishing sequence (as $n\to +\infty$); for all $n\in\N$, set $\Lambda_n:=\frac{1}{\delta_n}$, let $T_n$ be of the form \eqref{formT} (with $\delta$ replaced by $\delta_n$ and $\Lambda$ replaced by $\Lambda_n$)  and  let  
$u_n$ be a minimizer of $\F^{T_n}_{s}$ in $\A^{T_n,\Lambda_n,\delta_n}$. Then
$$
\lim_{n\to +\infty} \frac{1}{\sigma_n} \F^{T_n}_{s}(u_n) = 1,
$$
where 
\begin{equation}
\sigma_n:=\left\{\begin{array}{ll}
\displaystyle  \log\frac{1}{\delta_n}&\textrm{if }s=\frac 12\\ 
\displaystyle \frac{\delta_n^{1-2s}}{2s(1-s)(2s-1)(3-2s)}&\textrm{if }\frac 1 2<s<1.
\end{array}\right.
\end{equation}
\end{theorem}
%%%%%%%%%%%%%%%%%%%%%%%%
%%%%%%%%%%%%%%%%%%%%%%%%
%%%%%%%%%%%%%%%%%%%%%%%%
\begin{proof}
In view of  Theorem \ref{mainthm}, up to translations we may assume that the minimizers  $u_n$ are the functions defined by \eqref{defmin} for $\Lambda=\Lambda_n$ and $\delta=\delta_n$.
We start by observing that, since $T_n=L(1+\delta_n)$ with $L\in\N$,
\begin{equation}\label{primacosa}
(1+\delta_n) \F^{T_n}_s(u_n)=\frac 1 2\int_{-\delta_n}^{1}\ud x\int_{\R}\frac{|u_n(x)-u_n(y)|^2}{|x-y|^{1+2s}}\ud y.
\end{equation}
For every pair of measurable sets $(I,J)$, we define 
$$
\mathcal I^n_s(I,J):=\frac 1 2\int_{I}\ud x\int_{J}\frac{|u_n(x)-u_n(y)|^2}{|x-y|^{1+2s}}\ud y.
$$
Then, by \eqref{primacosa}, we have
\begin{equation}\label{secondacosa}
\begin{aligned}
(1+\delta_n) \F^{T}_s(u_n)=&\mathcal I^n_s((-\delta_n,0),(-\delta_n,0))+\mathcal I^n_s((-\delta_n,0),(0,1)) \\
&+\mathcal I^n_s((-\delta_n,0),(-1-\delta_n,-\delta_n)) +\mathcal I^n_s((-\delta_n,0),\R\setminus(-1-\delta_n,1))\\
&+\mathcal I^n_s(0,1),(-\delta_n,0))+\mathcal I^n_s((0,1),(0,1))+\mathcal I^n_s((0,1),(1,1+\delta_n))\\
&+
\mathcal I^n_s((0,1),(1+\delta_n,2+\delta_n))+\mathcal I^n_s((0,1),(-1-\delta_n,-\delta_n))\\
&+\mathcal I^n_s((0,1),\R\setminus(-1-\delta_n,2+\delta_n))\\
=:&\mathscr{I}^n_{1}+\mathscr{I}^n_{2}+\mathscr{I}^n_{3}+\mathscr{I}^n_{4}+\mathscr{I}^n_{5}+\mathscr{I}^n_{6}+\mathscr{I}^n_{7}+\mathscr{I}^n_{8}+\mathscr{I}^n_{9}+\mathscr{I}^n_{10}.
\end{aligned}
\end{equation}
We first discuss the case $s=\frac 1 2$. 
To this end, we observe that, for $n$ large enough,
\begin{equation}\label{terzacosa}
\mathscr{I}^n_{k}\le 1\qquad\textrm{for }k\neq 8,9.
\end{equation}
%Notice that
%\begin{equation*}
%\begin{aligned}
%\int_0^1\ud x\int_0^1\frac{|u_n(x)-u_n(y)|^2}{|x-y|^{2}}\ud y=&1\,,\\
%\int_{-\delta_n}^0\ud x\int_{-\delta_n}^0\frac{|u_n(x)-u_n(y)|^2}{|x-y|^{2}}\ud y=&\Lambda_n^2\delta_n^2=1\,,\\
%\int_{-\delta_n}^0\ud x\int_{0}^{\delta_n}\frac{|u_n(x)-u_n(y)|^2}{|x-y|^{2}}\ud y\le&\Lambda_n^2\delta_n^2=1\,,\\
%\int_{-\delta_n}^0\ud x\int_{\delta_n}^{1}\frac{|u_n(x)-u_n(y)|^2}{|x-y|^{2}}\ud y\le&1\,,\\
%\int
%\end{aligned}
%\end{equation*}
Now, by periodicity, 
\begin{equation}\label{terzacosa1}
\mathscr{I}^n_{8}=\mathscr{I}^n_{9}.
\end{equation}
For every $0<\rho<\frac 1 2$ we have that
\begin{equation*}%\label{contoprinc0}
\begin{aligned}
\mathscr{I}^n_{8}\ge &\frac 1 2(1-2\rho)^2\int_{1-\rho}^1\ud x\int_{1+\delta_n}^{1+\delta_n+\rho}\frac{1}{(y-x)^2}\ud y\\
%%=&\frac 1 2(1-2\rho)^2\int_{1-\rho}^1\ud x\int_{1+\delta_n}^{1+\delta_n+\rho}\frac{1}{(y-x)^{2}}\ud y\\
%= &-\frac 1 2(1-2\rho)^2\int_{1-\rho}^1\ud x\Big(\big(1+\delta_n+\rho-x\big)^{-1}-(1+\delta_n-x)^{-1}\Big)\\
%=&\frac 1 2(1-2\rho)^2\Big(2\log(\delta_n+\rho)-\log\delta_n-\log(\delta_n+2\rho)\Big)\\
=&\frac 1 2(1-2\rho)^2\Big(\log\frac{1}{\delta_n}+\log\frac{(\delta_n+\rho)^2}{\delta_n+2\rho}\Big),
\end{aligned}
\end{equation*}
whence, dividing by $\sigma_n$, sending first $n\to +\infty$ and then $\rho\to 0$, we obtain
\begin{equation}\label{contoprinc20}
\liminf_{n\to +\infty}\frac{1}{\sigma_n}\mathscr{I}^n_{8}\ge \frac 1 2.
\end{equation}
Moreover,  we have that
\begin{equation*}%\label{contoprinc200}
%\begin{aligned}
\mathscr{I}^n_{8}\le\frac 1 2\int_{0}^{1}\ud x\int_{1+\delta_n}^{2+\delta_n}\frac{1}{(y-x)^{2}}\ud y
%\le&\frac 1 2\int_{0}^{1-\rho}\ud x\int_{1+\delta_n}^{2+\delta_n}\frac{1}{(y-x)^{2}}\ud y+\frac 1 2\int_{1-\rho}^1\ud x\int_{1+\delta_n}^{1+\delta_n+\rho}\frac{1}{(y-x)^{2}}\ud y\\
%&+\frac 1 2\int_{1-\rho}^{1}\ud x\int_{1+\delta_n+\rho}^{2+\delta_n}\frac{1}{(y-x)^{2}}\ud y\\
=\frac 1 2\Big(\log\frac{1}{\delta_n}+\log\frac{(\delta_n+1)^2}{\delta_n+2}\Big),
%\frac{1}{\rho^2}+\frac{\sigma_n}{2}+\frac 1 2\log\frac{(\delta_n+\rho)^2}{\delta_n+2\rho},
%\end{aligned}
\end{equation*}
whence we obtain
\begin{equation}\label{contoprinc200}
\limsup_{n\to +\infty}\frac{1}{\sigma_n}\mathscr{I}^n_{8}\le \frac 1 2.
\end{equation}
By \eqref{primacosa}, \eqref{secondacosa}, \eqref{terzacosa}, \eqref{terzacosa1}, \eqref{contoprinc20} and \eqref{contoprinc200}, we obtain the claim for $s=\frac 1 2$.
\vskip5pt
We pass to the case $s>\frac 12 $. 
We first notice that, for $n$ large enough, 
\begin{equation}\label{limitati}
\mathscr{I}_4^n,\mathscr{I}_{10}^n\le 1;
\end{equation}
moreover, 
\begin{equation}\label{vamalecons}
%\begin{aligned}
\mathscr{I}_6^n=\frac 1 2\int_0^1\ud x\int_0^1|x-y|^{1-2s}\ud y
%=\frac 1 2\frac{1}{2-2s}\int_0^1\big(x^{2-2s}+(1-x)^{2-2s}\big)\ud x\\
=\frac{1}{2(1-s)(3-2s)}.\\
%\end{aligned}
\end{equation}
Now, as for $\mathscr{I}^n_1$, we have
\begin{equation}\label{Iuno}
%\begin{aligned}
\mathscr{I}^n_1=\frac{\Lambda_n^2}{2}\int_{-\delta_n}^{0}\ud x\int_{-\delta_n}^{0}|x-y|^{1-2s}\ud y
%=&\frac{\Lambda_n^2}{2(2-2s)}\int_{-\delta_n}^{0}\big((x+\delta_n)^{2-2s}+(-x)^{2-2s}\big)\ud x\\
%=&\frac{\Lambda_n^2}{2(1-s)(3-2s)}\delta_n^{3-2s}
=\frac{\delta_n^{1-2s}}{2(1-s)(3-2s)}.
%\end{aligned}
\end{equation}
As for $\mathscr{I}^n_2$, we claim that
\begin{equation}\label{Idueclaim}
\lim_{n\to +\infty}\frac{1}{\delta_n^{1-2s}}\mathscr{I}^n_{2}= \frac{1}{4s(3-2s)}.
\end{equation}
To this purpose,
we write
\begin{equation}\label{Idue}
\begin{aligned}
\mathscr{I}^n_2=&\frac 1 2\int_{-\delta_n}^{0}\ud x\int_{0}^{1}\frac{|\Lambda_n x+y|^2}{|x-y|^{1+2s}}\ud y
\\
=&\frac 1 2\int_{-\delta_n}^{0}\ud x\int_{0}^{1}\frac{|(\Lambda_n +1)x+y-x|^2}{|x-y|^{1+2s}}\ud y
\\
=&\frac 1 2\int_{-\delta_n}^{0}\ud x\int_{0}^{1}\frac{(y-x)^2}{(y-x)^{1+2s}}\ud y\\
&+\frac 1 2 \int_{-\delta_n}^{0}\ud x\int_{0}^{1}\frac{(\Lambda_n+1)^2x^2}{(y-x)^{1+2s}}\ud y\\
&+\int_{-\delta_n}^{0}\ud x\int_{0}^{1}\frac{(y-x)(\Lambda_n+1)x}{(y-x)^{1+2s}}\ud y\\
=:&\mathscr{I}^n_{2,1}+\mathscr{I}^n_{2,2}+\mathscr{I}^n_{2,3}.
\end{aligned}
\end{equation}
Straightforward computations yield
\begin{equation}\label{Idueuno}
\mathscr{I}^n_{2,1}\le \frac{\delta_n}{4(1-s)}.
\end{equation}
Moreover,
 \begin{equation*}%\label{Iduetre}
 \begin{aligned}
\mathscr{I}^n_{2,2}=&\frac{(\Lambda_n+1)^2}{4s}\int_{-\delta_n}^{0}(-x)^{2-2s}\ud x -\frac{(\Lambda_n+1)^2}{4s}\int_{-\delta_n}^{0}x^2(1-x)^{-2s}\ud x\\
=&\frac{1}{4s(3-2s)}(\Lambda_n+1)^2\delta_n^{3-2s}-\frac{(\Lambda_n+1)^2}{4s}\int_{-\delta_n}^{0}x^2(1-x)^{-2s}\ud x;
\end{aligned}
\end{equation*}
therefore,
using that
\begin{equation*}
\int_{-\delta_n}^{0}x^2(1-x)^{-2s}\ud x\le\delta^3_n, 
\end{equation*}
we deduce
\begin{equation}\label{Iduedue}
\lim_{n\to +\infty}\frac{1}{\delta_n^{1-2s}}\mathscr{I}^n_{2,2}= \frac{1}{4s(3-2s)}.
\end{equation}
Furthermore, using \eqref{Idueuno} and \eqref{Iduedue}, by H\"older inequality we deduce that
 \begin{equation}\label{Iduetre}
\lim_{n\to +\infty}\frac{1}{\delta_n^{1-2s}}\mathscr{I}^n_{2,3}=0\, .
\end{equation}
Therefore, the claim \eqref{Idueclaim} follows by \eqref{Idue}, \eqref{Idueuno}, \eqref{Iduedue}, and \eqref{Iduetre}.

By easy reflection/symmetry arguments, for every $n\in\N$,
\begin{equation}\label{quelliuguali}
\mathscr{I}^n_{2}=\mathscr{I}^n_{3}=\mathscr{I}^n_{5}=\mathscr{I}^n_{7}\qquad\qquad\textrm{and}\qquad\qquad \mathscr{I}^n_{8}=\mathscr{I}^n_{9}.
\end{equation} 
Therefore, it remains to compute
\begin{equation*}
\lim_{n\to +\infty}\frac{1}{\delta_n^{1-2s}}\mathscr{I}^n_{8}.
\end{equation*}
To this end, we notice that, for every $0<\rho<\frac 1 2$,
\begin{equation}\label{contoprinc1}
\begin{aligned}
\mathscr{I}^n_{8}
\ge&\frac 1 2(1-2\rho)^2\int_{1-\rho}^1\ud x\int_{1+\delta_n}^{1+\delta_n+\rho}\frac{1}{(y-x)^{1+2s}}\ud y\\
%= &-\frac{1}{4s}(1-2\rho)^2\int_{1-\rho}^1\ud x\Big(\big(1+\delta_n+\rho-x\big)^{-2s}-(1+\delta_n-x)^{-2s}\Big)\\
=&\frac{1}{4s(2s-1)}(1-2\rho)^2\Big(\delta_n^{1-2s}+(\delta_n+2\rho)^{1-2s}-2(\delta_n+\rho)^{1-2s}\Big),
%\\
%=&\frac{\delta_n^{1-2s}}{4s(2s-1)}(1-2\rho)^2\Big(1+\Big(1+2\frac{\rho}{\delta_n}\Big)^{1-2s}-2\Big(1+\frac{\rho}{\delta_n}\Big)^{1-2s}\Big),
\end{aligned}
\end{equation}
whence, sending first $n\to +\infty$ and then $\rho\to 0$, we deduce
\begin{equation}\label{contoprinc100}
\liminf_{n\to +\infty}\frac {1}{\delta^{1-2s}_n}\mathscr{I}^n_{8}\ge \frac{1}{4s(2s-1)}.
\end{equation}
Moreover,  we have that
\begin{equation*}%\label{contoprinc200}
%\begin{aligned}
\mathscr{I}^n_{8}\le\frac 1 2\int_{0}^{1}\ud x\int_{1+\delta_n}^{2+\delta_n}\frac{1}{(y-x)^{1+2s}}\ud y
%\le&\frac 1 2\int_{0}^{1-\rho}\ud x\int_{1+\delta_n}^{2+\delta_n}\frac{1}{(y-x)^{2}}\ud y+\frac 1 2\int_{1-\rho}^1\ud x\int_{1+\delta_n}^{1+\delta_n+\rho}\frac{1}{(y-x)^{2}}\ud y\\
%&+\frac 1 2\int_{1-\rho}^{1}\ud x\int_{1+\delta_n+\rho}^{2+\delta_n}\frac{1}{(y-x)^{2}}\ud y\\
=
\frac{1}{4s(2s-1)}\Big(\delta_n^{1-2s}+(\delta_n+2)^{1-2s}-2(\delta_n+1)^{1-2s}\Big),
%\frac{1}{\rho^2}+\frac{\sigma_n}{2}+\frac 1 2\log\frac{(\delta_n+\rho)^2}{\delta_n+2\rho},
%\end{aligned}
\end{equation*}
whence we obtain
\begin{equation}\label{contoprinc1000}
\limsup_{n\to +\infty}\frac{1}{\delta_n^{1-2s}}\mathscr{I}^n_{8}\le \frac {1}{4s(2s-1)}.
\end{equation}
By \eqref{contoprinc100} and \eqref{contoprinc1000}, we get
\begin{equation}\label{limiteIotto}
\lim_{n\to +\infty}\frac{1}{\delta_n^{1-2s}}\mathscr{I}^n_{8}= \frac {1}{4s(2s-1)}.
\end{equation}
Finally, by \eqref{secondacosa}, \eqref{limitati}, \eqref{vamalecons}, \eqref{Iuno}, \eqref{Idueclaim}, \eqref{quelliuguali}, and \eqref{limiteIotto}, we get the claim.
\end{proof}

%%%%%%%%%%%%%%%%%%%%%%%%
%%%%%%%%%%%%%%%%%%%%%%%%
%%%%%%%%%%%%%%%%%%%%%%%%

%%%%%%%%%%%%%%%%%%%%%%%%
%%%%%%%%%%%%%%%%%%%%%%%%
%%%%%%%%%%%%%%%%%%%%%%%%

\subsection{The extremal cases $s=1$ and $s=0$}\label{estremi}
Here we discuss the extremal cases  $s=1$ and $s=0$.
As for $s=1$,  we are led to consider   \cite{BBM, MS, CDKNP} the energy functional
\begin{equation}\label{defFs1}
\F^{T}_1(u):=\frac{1}{2T}\int_{0}^T |  u'(x)|^2 \ud x .
\end{equation} 
%The case $s=1$ is trivial; in this case we consider  the energy functional 
%\begin{equation}\label{defFs1}
%\F^{T}_1(u):=\frac{1}{2T}\int_{0}^T |  u'(x)|^2 \ud x  .
%\end{equation}
The functional $\F^{T}_1$ is constant on $\A^{T,\Lambda,\delta}$, and hence every $u\in \A^{T,\Lambda,\delta}$ is a minimizer of $\F^{T}_1$ in  $\A^{T,\Lambda,\delta}$ . We also notice that, if in the definition of  $\A^{T,\Lambda,\delta}$ we replace the condition that $u$ is $T$-periodic with the $T$-periodicity of $u'$, then the  minimizers of $\F^{T}_1$ become  the functions $u(x) = x+c, \, c\in \R$ if $1\le \Lambda$, 
 $u(x) = - \Lambda x+c, \, c\in \R$ if $1\ge  \Lambda$.
 
 The case $s=0$ corresponds  \cite{MS, CDKNP} to the functional

\begin{equation}\label{defFs0}
\F^{T}_0(u):=\frac{1}{2T}\int_{0}^T   |u(x)|^2 \ud x  .
\end{equation}

The periodicity of the minimizers of $\F^T_0$ has been proven in \cite{M93} for $\Lambda=1$ (and extended in \cite{RW} to all $\Lambda>0$), in a less rigid setting, namely, replacing the condition $u'\in\{\pm 1\}$ with a Modica-Mortola penalization, whose parameter $\e>0$  implicitly fixes the transition scale ($\delta$, in our model). Theorem \ref{s0} below concerns with the periodicity of the minimizers of $\F^T_0$ in $\A^{T,\Lambda,\delta}$.

\begin{theorem}\label{s0}
The minimizers of $\F^{T}_0$ in $\A^{T,\Lambda,\delta}$ are the functions $\bar u(\cdot - x_0) - \frac{\Lambda \delta}{2}$, with $x_0\in\R$.
\end{theorem}

\begin{proof}
Let $u$ be a minimizer of $\F^{T}_s$ in $\A^{T,\Lambda,\delta}$ and let $(x_1,x_2)$ be a connected component of $(u')^{-1}(\{-\Lambda\})$. 
Let, for $h>0$ small enough, 
$u_h= u+h \varphi_h \in \A^{T,\Lambda,\delta}$  for some functions $\varphi_h$ satisfying $\varphi_h \to (\Lambda +1)  \chi_{\{(x_2-\delta,x_2)) + kT, \, k\in\Z\}}$ strongly in $L^1_{\loc}(\R)$ as $h\to 0^+$. Then, by minimality of $u$ we have
$$
0\le T \lim_{h\to 0} \frac{\F^{T}_0(u_h) - \F^{T}_0(u)}{h} = (\Lambda +1) \int_{x_2-\delta}^{x_2}  u(x) \ud x. 
$$
Analogously, we have 
$$
0\le  - (\Lambda +1) \int_{x_1}^{x_1+\delta}  u(x) \ud x. 
$$
We deduce that
$$
0 \le \int_{x_2-\delta}^{x_2}  u(x) \ud x \le \int_{x_1}^{x_1+\delta}  u(x) \ud x \le 0.
$$
Therefore, $x_2=x_1+\delta$ and $u(x_2)= - u(x_1)$, and such relations hold true  for the end points of  all the connected components of $(u')^{-1}(\{-\Lambda\})$. This is equivalent to the fact that $u = \bar u(\cdot - x_0) - \frac{\Lambda \delta}{2}$ for some  $x_0\in\R$. 
\end{proof}

Finally, in the following proposition we determine the asymptotic behavior of the minimal energy $\F^T_s$ for $s=0,1$.
\begin{proposition}
Let  $\{\delta_n\}_{n\in \N}$ be a positive vanishing sequence (as $n\to +\infty$); for all $n\in\N$, set $\Lambda_n:=\frac{1}{\delta_n}$, let $T_n$ be of the form \eqref{formT} (with $\delta$ replaced by $\delta_n$ and $\Lambda$ replaced by $\Lambda_n$)  and  let  
$u_n$ be a minimizer of $\F^{T_n}_{0}$ in $\A^{T_n,\Lambda_n,\delta_n}$. Then
$$
\lim_{n\to +\infty}  \F^{T_n}_{0}(u_n) = \frac{1}{24}.
$$
Moreover, 
let  
$u_n$ be a minimizer of $\F^{T_n}_{1}$ in $\A^{T_n,\Lambda_n,\delta_n}$. Then
$$
 \lim_{n\to +\infty} \delta_n  \F^{T_n}_{1}(u_n) = \frac 12.
$$
\end{proposition}

%%%%%%%%%%%%%%%%%%%%%%%%%%%%
%%%%%%%%%%%%%%%%%%%%%%%%%%%%
%%%%%%%%%%%%%%%%%%%%%%%%%%%%

\section{van der Merwe formula for misfit dislocations between two half-planes with different elastic coefficients}\label{Sdue}
In this section we provide the energy density induced by misfit dislocations at semi-coherent straight interfaces  separating two elastic materials with different elastic moduli and filled with square lattices having the same orientation but different lattice spacing. 
\subsection{Elastic energy  on the half-plane}
Here we introduce well known facts about half-plane isotropic linearized elasticity, focussing on the energy induced by a given  datum on the boundary of the half-plane.

Let $G>0$ and $-\frac 12 <\nu <1$ denote the shear modulus and the Poisson ratio of the elastic material occupying the half plane $\Hp^+:= \R \times \R^+ $, respectively. Given $U\in H^1(\Hp^+;\R^2)$ we denote by $\es(U):= \frac 12 (\nabla U + \nabla^{\mathrm T} U)$ the symmetric gradient of $U$;  the corresponding isotropic elastic energy is given by 
$$
E^{\el}(U):= \frac 12 \int_{\Hp^+}  \frac{2G \nu}{1-2\nu} |\Tr (\es(U))|^2 +  2 G |\es(U)|^2 \ud \mathsf{x}= 
\frac 12 \int_{\Hp^+} \sigma(U): \es(U) \ud\mathsf{x}
$$
where $\sigma(U):=  \frac{2G \nu}{1-2\nu} \Tr(\es(U)) \mathrm{Id} + 2G \es(U)$ is the stress associated to $U$. 

Set $\Hp^0:= \R \times \{0\}$. Given $u\in H^{\frac 12} (\Hp^0)$ we denote by $U$ the elastic extension of $(u,0)$ on $\Hp^+$,  namely the function $U$ minimizing $E^{\el}(U)$ among all $H^1$-functions whose trace at $\Hp^0$ equals to $(u,0)$. Equivalently, $U$ solves 
\begin{equation}\label{ee}
\left\{
\begin{array}{ll}
\mathrm{Div} \, \sigma(U)=0&\textrm{in }\Hp^+\\
U=(u,0)&\textrm{on }\Hp^0.
\end{array}
\right.
\end{equation}

Denoting by $x$ and $x_2$ the horizontal and the vertical coordinates of $\mathsf{x} \in \R^2$, respectively,
integration by parts yields
\begin{equation}\label{ibp}
E^{\el}(U) = \frac 12 \int_{\Hp^0} u(x) \cdot  \sigma_{12}(U)  \ud x,
\end{equation}
where   $\sigma_{12}(U) $ at $\Hp^0$ belongs to  $H^{-\frac 12}$, so that its  product with $u$ is meant  in the sense of duality between $H^{\frac 12}$ and $H^{- \frac 12}$. 
The map $u\mapsto \sigma_{12}(U)$ is a Dirichlet-to-Neumann map; for $u$ regular enough it can be explicitly computed by means of Fourier analysis methods and it is given by (see \cite[Lemma 2.3]{GLLX}) 
$$
\sigma_{12}(U)  (x):= -\frac{G}{(1-\nu) \pi} \lim_{r\to 0^+} \int_{\R\setminus (x-r,x+r)} \frac{u(x)-u(y)}{|x-y|^2} \ud y \qquad \text{ for all } x\in\R. 
$$
By \eqref{ibp} we deduce that for all $u$ smooth enough, and in fact, by an easy density argument, for all $u\in H^{\frac 12} (\Hp^0)$, 
\begin{equation}\label{ff}
E^{\el}(U) = \frac{G}{4\pi(1-\nu)}  \int_\R  \ud x \int_\R \frac{|u(x)- u(y)|^2}{|x-y|^2} \ud y \, .
%\|u\|_{\dot H^{\frac 12}}^2\, .
\end{equation}

Assume now that $u\in H^{\frac 12}_{\loc}(\R)$ is $T$-periodic. Let $V$ be the unique minimizer of the elastic energy 
\begin{equation}\label{ff3}
E^{\el}(V; (0,T)\times \R^+) := \frac 12 \int_{(0,T)\times \R^+} \sigma(V): \es(V) \ud \mathsf{x},
\end{equation}
among functions $V\in H^1((0,T)\times \R^+;\R^2)$ with $V=(u,0)$ on $(0,T)\times \{ 0 \}$ and with $ V(0,\cdot)= V(T,\cdot)$ in the sense of traces. By minimality $V$ is traction free on the vertical  half lines $\{0\}\times \R^+$ and $\{T\}\times \R^+$. 
Let moreover $U$ be the  $T e_1$-periodic extension of $V$, namely such that 
$U(x,x_2) = U(x+T,x_2)$ for almost all  $x\in\R,\, x_2\in\R^+$. Then, $U$ solves the elasticity equations 
\eqref{ee}. Furthermore, by \eqref{ff} and by standard cut-off arguments it follows that 

\begin{equation}\label{ff2}
E^{\el}(U; (0,T)\times \R^+) =
 \frac{G}{2\pi (1-\nu)}T\F^T_{\frac 1 2}(u),
\end{equation}
where $\F^T_{\frac 1 2}$ is the functional defined in \eqref{defF}.

\subsection{van der Merwe formula} 
Let $\Hp^\pm:= \R\times \R^\pm$ and assume that $\Hp^\pm$ behave like linearly elastic isotropic materials with elastic coefficients $(G^\pm,\nu^\pm)$. From a microscopic perspective, assume that $\Hp^\pm$ are filled by square lattices with lattice spacing given by $c^\pm$ with $c^-\ge c^+>0$, and set $c:=\frac{c^+ + c^-}{2}$.

Let $U^\pm\in H^1_{\loc}(\Hp^\pm)$ be elastic displacements defined on $\Hp^\pm$ with trace on $\Hp^0:= \R \times \{0\}$ having only horizontal component,   denoted by $u^\pm \in H^{\frac 12}_{\loc}(\Hp^0)$.
%  the (horizontal) traces of $U^\pm$ on $\Hp^0$. 

 We start by describing the structural assumptions of our model.  We assume that $\Hp^0$ is divided into two kinds of regions: an elastic one, where the top and bottom layer perfectly match their lattices through uniform elastic deformations, and a plastic region, where  the lattice misfits are compensated by the emergence of edge dislocations. 
We also assume that on each of these regions the deformation gradient of $u^\pm$  are constants, denoted by $m^\pm_{\mathrm e}$ and $m^\pm_{\mathrm p}$ on the elastic and plastic regions, respectively; these constants  will be determined by energy minimization principles, together with linearization arguments based on the assumption that the interface is semi-coherent, i.e., $\frac{c^+}{c^-} \approx 1$. 
 
Now, we compute the displacement gradients on the elastic and plastic regions for $u^\pm$. 
In the elastic region, enforcing a perfect matching between the top and the bottom lattices we get the relation 
$$
\frac{1+m_{\mathrm e}^+}{1+m_{\mathrm e}^-} = \frac{c^-}{c^+}.
$$ 
Linearizing such a relation around $c^+ = c^-$  and hence $m_{\mathrm e}^\pm = 0$  we get
\begin{equation}\label{rel1}
m_{\mathrm e}^+ - m_{\mathrm e}^-= \frac{c^- - c^+}{c} =:m. 
\end{equation}
Here $m$ is a material constant representing the purely elastic strain  needed to produce a perfect lattice matching. Clearly,  $m$  may be distributed on $m_{\mathrm e}^\pm$ through an auxiliary parameter $\alpha_{\mathrm e}\in[0,1]$ such  that $m^+_{\mathrm e} = \alpha_{\mathrm e} m$, $m^-_{\mathrm e} = - (1-\alpha_{\mathrm e}) m$. 
This  condition can be also expressed in terms of $m^\pm_{\mathrm e}$ as
\begin{equation}\label{rel11}
(1-\alpha_{\mathrm e}) m^+_{\mathrm e} = - \alpha_{\mathrm e}  m^-_{\mathrm e}.
\end{equation}  

Let us focus now on the plastic region. We assume that  each edge dislocation corresponds to a region whose length is given by few lattice spacings where large deformations take place. In such a region, we deform a bottom string of $K$ atoms in order that its extremal atoms match the corresponding top extremal atoms of a string of $K+1$ atoms. The parameter $K$ represents the length of the core region in terms of multiples of the lattice spacing. In fact, while the optimal profile around an edge dislocation has not compact support, the analysis  by van der Merwe and Peierls-Nabarro shows that, up to a small tail, the profile is essentially concentrated on the scale of the lattice spacing. Therefore, in our simplified model we assume  that  the profile is supported on a core region of size between $K c$ and $(K+1) c$ for some $K\in\N$. Moreover, the specific choice of $K$ and the sharp profile of the interface will play no role in our analysis, since they affect the asymptotic behavior of the energy density only by lower order terms. Hence,  without loss of generality, we  fix   $K=2$ and  the profile of the interface to be affine. 

The picture is the following: Consider three consecutive atoms on the top layer lying on the top of two consecutive atoms of the bottom layer (so that the central atom on the top is equidistant to the two atoms at the bottom); enforcing that the two external atoms on the top perfectly match the two atoms on the bottom, we get the  following 
relation 
\begin{equation}\label{nolin}
\frac{1+m^+_{\mathrm p}}{1+m^-_{\mathrm p}} = \frac{c^-}{2c^+}.
\end{equation}
As  done in \eqref{rel11} for the elastic strains,  we enforce that 
\begin{equation}\label{nolin2}
(1-\alpha_{\mathrm p})m^+_{\mathrm p}=-\alpha_{\mathrm p} m^-_{\mathrm p}
%u^+ = -\alpha w, \, u^- = (1-\alpha)w \qquad 
\qquad\text{ for some  }% w \text{ and some } 
\alpha_{\mathrm p} \in [0,1].
\end{equation}
Then, %linearizing \eqref{nolin}  around
taking  $c^\pm=c$ in  \eqref{nolin}  and using \eqref{nolin2},  we get
\begin{equation}\label{rel2}
m^+ _{\mathrm p}- m^-_{\mathrm p}=  - \frac{1}{\alpha_{\mathrm p} + 1} \, . 
\end{equation}
Notice that $\alpha_{\mathrm p}=1$ corresponds to  configurations  where the bottom layer is rigid, and gives back $m^+ _{\mathrm p}- m^-_{\mathrm p}= -\frac 12$ which is consistent with the analysis done in \cite{FPS}. Set $u:= u^+ - u^-$;
taking into account the relations \eqref{rel1} and \eqref{rel2}, we have   that 
\begin{equation}\label{oran}
u '\in\Big\{m, - \frac{1}{\alpha_{\mathrm p} + 1}  \Big\}\, .
\end{equation}

Now we introduce the notion of core radius $\e>0$ in our model. To this purpose, we assume that, in a reference configuration,  the plastic region (namely the region where $u' = - \frac{1}{\alpha_{\mathrm p} + 1}$)  is given by the union of disjoint intervals of size $\e$;
such a parameter represents the core radius of the dislocation, and is 
determined by the following reasoning: Set $b^+ := -m_{\mathrm p}^+ \e, \, b^-:= m_{\mathrm p}^- \e$; such quantities represent fractional top and bottom portions of the Burgers vector, so that their sum is equal to $|b|:=c$. In view of \eqref{rel2} we deduce $\e= c(\alpha_{\mathrm p} +1)$.
   
Finally, we enforce that the top elastic and plastic regions  agree with the bottom ones,  and that $u^+$ and $u^-$ are periodic functions (a priori with arbitrarily, possibly different, periods). Let $T>0$ and let $l_{\mathrm e}(T)$ and $l_{\mathrm p}(T)$ be the lengths of the elastic and plastic regions in the interval $[0,T]$. Then, in the limit as $T\to +\infty$ we easily deduce 
$$
-\frac{m_{\mathrm p}^\pm}{m_{\mathrm e}^\pm}= \lim_{T\to +\infty }\frac{l_{\mathrm e}(T)}{l_{\mathrm p}(T)}. 
$$
(Periodicity  ensures that the above limit exists.) 
Therefore, by \eqref{rel11} and \eqref{nolin2}, we get
$$
\frac{\alpha_{\mathrm p}}{1-\alpha_{\mathrm p} }= - \frac{m_{\mathrm p}^+}{m_{\mathrm p}^-} = - \frac{m_{\mathrm e}^+}{m_{\mathrm e}^-} = \frac{\alpha_{\mathrm e}}{1-\alpha_{\mathrm e}},
$$
and hence $\alpha_{\mathrm p}\equiv\alpha_{\mathrm e}=:\alpha$.
% we assume that $\alpha_{\mathrm p}\equiv\alpha_{\mathrm e}=:\alpha$, and that
Notice that, up to translations of the reference top and bottom lattices, we may always assume that    $u^+(\bar x)=u^-(\bar x)=0$ at some point $\bar x$.  Then, by \eqref{rel11}, \eqref{nolin2} and \eqref{oran}, we have 
$$
u^+ = \alpha u, \qquad\qquad u^- = -(1-\alpha)u
$$
for some $T$-periodic function $u$ with  
%In what follows we will consider $\alpha$ and $\alpha$ as independent parameters; we will optimize with respect to $\alpha$ and show that the result of such an optimization does not depend on $\alpha$.

%Enforcing that $u$ is $T$-periodic and plugging the conditions already detected for $u^+ - u^-$ we get that  
$$
\frac{1}{m} u \in \A^{T, \frac{1}{(\alpha + 1)m} ,(\alpha+1)c},
$$
and $T$ of the form \eqref{formT} (for $\Lambda= \frac{1}{(\alpha + 1)m}$ and $\delta=(\alpha+1)c$).
Setting $E^{\mathrm{tot}}:= E^{\el}(U^+; (0,T)\times \R^+)  + E^{\el}(U^-; (0,T)\times \R^-)$, by  \eqref{ff2} we have
\begin{equation}\label{ff3}
E^{\mathrm{tot}} =
T\F^T_{\frac 1 2}(u)  \left(  \alpha^2   \frac{G^+}{2\pi (1-\nu^+)}   + (1-\alpha)^2 \frac{G^-}{2\pi (1-\nu^-)}\right) .
\end{equation}
Minimizing $E^{\mathrm{tot}}$ with respect to $\alpha$ yields
\begin{equation}\label{amin}
\alpha_{\min}=
%\frac{\frac{G^-}{2\pi (1-\nu^-)}}{\frac{G^+}{2\pi (1-\nu^+)} + \frac{G^-}{2\pi (1-\nu^-)}}
%= 
\frac
{
G^-   (1-\nu^+)
}
{
G^+  (1-\nu^-)      + G^-   (1-\nu^+)}\, .
\end{equation}
Recalling Theorem \ref{mainthm},
minimizing \eqref{ff3} with respect to the functions $u$ such that $\frac{1}{m} u \in \A^{T,\frac{1}{(\alpha + 1)m} ,(\alpha+1)c}$ we get that the minimizer $u_{\min}$ is such that      $\frac{1}{m} u_{\min}$ 
is (up to translations) as in \eqref{defmin}, again  with $\Lambda$ 
replaced by $\frac{1}{(\alpha+1)m}$ and $\delta$ replaced by $(\alpha+1)c$. 

Let $v(x):= \frac{1}{c} u_{\min}(\frac{c}{m} x)$.
Then, 
$v$ is a minimizer of $\F^{T\frac m c}_{\frac 12}$ in $\A^{T\frac m c,\delta^{-1},\delta}$ for $\delta:= (\alpha+1)m$.

 Then, by \eqref{ff3}, Theorem \ref{pers>}, and \eqref{amin}, we get
\begin{multline*}
\frac{1}{T}E^{\mathrm{tot}} =  \left(  \alpha_{\min}^2   \frac{G^+}{2\pi (1-\nu^+)}   + (1-\alpha_{\min})^2 \frac{G^-}{2\pi (1-\nu^-)}\right) c \,m\, \F^{T\frac{m}{c}}_{\frac 1 2}(v) 
\\
=\left(  \alpha_{\min}^2   \frac{G^+}{2\pi (1-\nu^+)}   + (1-\alpha_{\min})^2 \frac{G^-}{2\pi (1-\nu^-)}\right) c \,m\,  \log \frac 1 m+ \textrm{l.o.t.}
\\
=\frac{1}{2\pi}  \frac{G^+G^-}{G^- (1-\nu^+)+G^+(1-\nu^-)}    (c^--c^+)\,  \log \frac c {c^- - c^+}+ \textrm{l.o.t.} \, .
\end{multline*}

Writing the surface energy in terms of the length of the Burgers vector $|b|=c$ and of the distance between dislocations $\Delta:= \frac{c}{m} + \e =\frac{c^2}{c^- - c^+} + \e$, we get 
\begin{multline}\label{gfvdm}
\frac{1}{T}E^{\mathrm{tot}}=  \frac{1}{2\pi}  \frac{G^+G^-}{G^- (1-\nu^+)+G^+(1-\nu^-)}   |b|^2 \frac{1}{\Delta-\e} \log \Big(\frac{\Delta-\e}{|b|}\Big)+ \textrm{l.o.t.}\\
=\frac{1}{2\pi}  \frac{G^+G^-}{G^- (1-\nu^+)+G^+(1-\nu^-)}   |b|^2 \frac{1}{\Delta} \log \Big(\frac{\Delta}{|b|}\Big)+ \textrm{l.o.t.},
\end{multline}
where we have used that  in the semi-coherent limit $\Delta \approx \frac{c}{m}$.
%Notice that in the semi-coherent limit we have $m \to 0$,  $\Delta \approx \frac{c}{m}$, and by l.o.t. we mean terms that, divided by
%$cm\log(\frac{1}{m})$, tend to zero as  $m \to 0$.
Formula \eqref{gfvdm} is consistent with the analysis of van der Merwe \cite[formula (26)]{VdM} for $G^+ = G^-$, $\nu^+ = \nu^-$ (and hence $\alpha_{\min} = \frac 12$) and it represents the analogous of the Shockley-Read \cite{RS} formula for misfit dislocations.

%%%%%%%%%%%%%%%%%%%%%%%
%%%%%%%%%%%%%%%%%%%%%%%
%%%%%%%%%%%%%%%%%%%%%%%
\section{Conclusions and perspectives}\label{Stre}

We have considered two slope functions with fixed length scale; we have proven periodicity of minimal configurations for the $s$-fractional Gagliardo seminorms with $s\in(0,1)$; then, we have considered  also  the extremal cases $s=0$ and $s=1$, corresponding to the $L^2$ norm and to the $\dot H^1$-seminorm, respectively. We have also provided the asymptotic 
behavior of the energy density as $\delta\to 0$, $\Lambda=\frac{1}{\delta}$. For $s=\frac 12 $ such a result is particularly relevant to compute surface energy densities stored at semi-coherent interfaces. Specifically, we have computed the surface energy induced by misfit dislocations between two linearly elastic half planes filled by semi-coherent square lattices, generalizing the van der Merwe analysis to the case where the two half planes have different elastic moduli. The main novelty of such a result is that we prove, rather than assume, the optimality of the periodic distribution of misfit dislocations.

\vskip3pt
\paragraph{\bf Generalized $s$-fractional seminorms for all values of $s$}
A natural follow-up of our analysis would be to consider the Gagliardo $s$-seminorms when $s$ is negative. For instance, for $s=-1$ the energy functional takes the form
$$
\F_{-1}^{T}(u):=\frac{1}{2T}\inf_{\newatop{v\in L^2(0,T)}{v'=u}}\|v\|^2_{L^2(0,T)}.
$$
For negative values of $s$, another possibility is to formally (or by means of $\Gamma$-convergence) renormalize the integrand in \eqref{defF} by expanding the square $|u(x)-u(y)|^2$ and  subtracting  the infinite $L^2$ contributions; similar renormalization procedures have been considered in \cite{DNP, CDKNP}, also for the critical case $s=0$. The resulting energy functional is, in general, non-positive and takes the form of a Riesz (or weighted $XY$) type functional, formally defined (on a suitable functional space to be specified) by
\begin{equation*}
\F_{s}^{T}(u):=-\frac{1}{T}\int_{0}^{T}\ud x\int_{\R}\frac{u(y)u(x)}{|x-y|^{1+2s}}\ud y.
\end{equation*}
The first variation of such  an energy gives back nonlocal differential operators that could be understood as generalized fractional Laplacians for negative $s$, and deserve, in our opinion, further investigation. 
On the other hand,  from our analysis, and formally writing $\|u\|_{\dot H^s}=\| u'\|_{\dot{H}^{s-1}}$,  it seems that a natural range for the parameter $s$ is given by  $0<s< \frac 32$, $s=1$ being a critical value where a renormalization procedure is needed. Another renormalization of the energy could take place and be analyzed at $s = \frac 32$, in order to extend the analysis to all positive values of $s$ (we refer to \cite{DKP} where it is shown that supercritical fractional seminorms of characteristic functions of sets behave, after suitable renormalization procedures, as the Euclidean perimeter of the sets). 

\vskip3pt
\paragraph{\bf Phase field models}
In this paper we have considered rigid affine  profiles where $u$ has a constant negative slope
$-\Lambda$ on essentially disjoint $\delta$-intervals of the type $(x_0-\frac\delta 2,x_0+\frac\delta 2)$. In fact, any given symmetric profile could be considered. More precisely, we can assume that $u(\cdot)= \psi(\cdot-x_0)$ on  $\delta$-intervals  of the type $(x_0-\frac\delta 2,x_0+\frac\delta 2)$,  where $\psi:(-\frac{\delta}{2}, \frac{\delta}{2}) \to \R$ is a given (smooth enough) non-increasing odd function. Periodicity of minimizers as in Theorem \ref{mainthm} would follow as well with minor changes. Such a  generalization would confirm that \eqref{gfvdm} holds true also prescribing more realistic (still with compact support)  plastic profiles for the traces of the displacements in the core regions of the misfit dislocations. 
   
One could also consider less rigid models, where the length scale is not quantized {\it a priori}. 
%and plugged in the class of admissible configurations
% but enforced by an extra small parameter $\e>0$.
In this respect, a first generalization of our model would consist in replacing the quantization constraint in the class of admissible configurations with a minimality condition on the length scale where the negative derivative is assumed.

A natural choice, commonly used in literature, would be to avoid any length scale constraint in the class of admissible configurations,
 but to enforce it by an extra small parameter $\e>0$. In  what follows we consider for simplicity one-periodic functions, i.e.,  we fix $T=1$. A basic energy functional could be %given by
$$
\G^{\e}_s(u):= \frac{1}{2}\int_{0}^{1}\ud x\int_{\R}\frac{|u(y)-u(x)|^2}{|x-y|^{1+2s}}\ud y + \e \sharp (S_{u'}\cap (0,1]),
$$ 
to be minimized among two slope functions without further restrictions. 
Furthermore, the transition between the two slopes could be relaxed considering a Modica-Mortola penalization such as
$$
\bar\G^{\e}_s(u):=  \frac{1}{2}\int_{0}^{1}\ud x\int_{\R}\frac{|u(y)-u(x)|^2}{|x-y|^{1+2s}}\ud y + \e^2 \int_0^1 |u''|^2 \ud x +  \int_0^1 W(u') \ud x\; , 
$$
%where $W$ is a double well potential with minima at the two desired slopes. This has been (up to minor differences) done  for two symmetric opposite slopes in \cite{M93} for the $L^2$ norm, corresponding to $s=0$, and in \cite{GM} for $s= \frac 12$. The case of two different slopes and $s= \frac12$ has been proposed in \cite{FPS}, together with a variant of the classical Modica-Mortola type functional, the difference being the presence of a  given constant eigenstrain.  Such a model is closely related to the Peierls-Nabarro formalism of dislocations; neglecting all the physical material parameters related to dislocations and using instead the set of parameters $(1,\Lambda, \delta, s)$, according with those appearing in Section \ref{Suno}, the energy functional reads as 
where $W$ is a double well potential with minima at the two desired slopes. This has been (up to minor differences) done for the $L^2$ norm, corresponding to $s=0$, in  \cite{M93}  for two symmetric opposite slopes and in \cite{RW} without any symmetry assumption; furthermore, the case  $s= \frac 12$ has been done in \cite{GM} for two symmetric opposite slopes.
% for two symmetric opposite slopes in \cite{M93} for the $L^2$ norm, corresponding to $s=0$, and in \cite{GM} for $s= \frac 12$. 

 The case of two different slopes and $s= \frac12$ has been already proposed   in \cite{FPS} as a model for misfit dislocations. There, uniform distribution of dislocations, rather that their periodicity, has been proved in the semi-coherent limit. In such a paper, another model 
closely related to the Peierls-Nabarro formalism of dislocations has been proposed; neglecting all the physical and material parameters related to dislocations and using instead the set of parameters $(1,\Lambda, \delta, s)$, according with those appearing in Section \ref{Suno}, the energy functional reads as 
$$
\mathcal{PN}^{\delta,\Lambda}_s(u):=\frac{1}{2} \int_{0}^{1}\ud x\int_{\R}\frac{|u(y)-u(x)|^2}{|x-y|^{1+2s}}\ud y + \frac{1}{\delta^2} \int_0^1 \mathrm{dist}^2(u(x) -x,\delta (1+\Lambda) \Z) \ud x.
$$    
Such a model differs from  classical Modica-Mortola functionals since the $\dot{H}^1$-term is replaced by the $s$-fractional seminorm and the potential has infinite periodic wells; such a variant was considered, for $s=\frac 12$,  in \cite{FG} in connection with the energy induced by dislocations  at coherent interfaces; a second novelty of $\mathcal{PN}^{\delta,\Lambda}_s(u)$ is due to the presence of the eigenstrain, proposed in \cite{FPS}, related to the presence of semi-coherent intefaces and possibly enforcing periodic distribution of dislocations.   
In fact, we expect that minimizers of such an energy functional exhibit the same behavior of those of $\F^1_s$ in $\A^{1,\Lambda,\delta}$. The analysis of $\mathcal{PN}_{s}^{\delta,\Lambda}$ and its material-dependent variants more closely related to specific dislocation misfit models (for $s=\frac12$) deserve, in our opinion,  further investigations.

%\begin{itemize}
%\item Sembra che il range naturale di $s$ sia $s<\frac32$. Il caso $s\ge \frac 32$ si potrebbe anche considerare per rinormalizzazione; forse conta solo il numero di salti?
%\item Sottraendo $C_s \|u\|_2^2$ dove $C_s$ e' la costante giusta di rinormalizzazione il risultato e' ancora vero? 
%\item $p\neq 2$
%\end{itemize}

%%%%%%%%%%%%%%%%%%%%%%%%%%%%
%%%%%%%%%%%%%%%%%%%%%%%%%%%%
%%%%%%%%%%%%%%%%%%%%%%%%%%%%

\end{document}